\documentclass{article}
\RequirePackage[OT1]{fontenc}
\RequirePackage{amsthm,amsmath,amssymb}
\RequirePackage[numbers]{natbib}

\newtheorem{proposition}{Proposition}[section]
\newtheorem{theorem}{Theorem}[section]
\newtheorem{definition}{Definition}[section]
\newtheorem{remark}{Remark}[section]
\newtheorem{corollary}{Corollary}[section]

\newtheorem{example}{Example}[section]
\numberwithin{equation}{section}
\theoremstyle{plain}

\begin{document}
\markright{Convergence of the sequences of Young measures}
\thispagestyle{empty}

\begin{center}
\textbf{\Large Weak convergence of the sequences of homogeneous Young measures associated with a class of oscillating functions}\\
\vspace{1cm}
Piotr Pucha{\l}a\\
\vspace{0.5cm}
{\small\mbox{Institute of
 Mathematics, Czestochowa University of Technology,
 al. Armii Krajowej 21}, 42-200 Cz\c{e}stochowa, Poland\\
Email: piotr.puchala@im.pcz.pl, p.st.puchala@gmail.com}
\end{center}

\begin{abstract}
We take under consideration Young measures with densities. The notion of density of a Young measure is introduced and illustrated with examples. It is proved that the density of a Young measure is weakly sequentially closed set. In the case when density of a Young measure is a singleton (up to the set of null measure), it is shown that the strong closedness (in $\textnormal{rca}(K)$) of the set of such measures, associated with Borel functions with values in the compact set $K\subset\mathbb R^l$, is equivalent with the strong closedness (in $L^1(K)$) of the set of their densities, provided the set $K$ is convex. For an m-oscillating function the notion of a total slope is proposed. It turns out, that if the total slopes of the elements of the sequence of oscillating functions form monotonic sequence, then the sequence of the respective (homogeneous) Young measures is weakly convergent in $\textnormal{rca}(K)$. The limit is a homogeneous Young measure with the density being the weak $L^1$ sequential limit of the densities of the underlying Young measures.\\

\noindent\textbf{Keywords}: Young measures, weak convergence, density, total slope \\
\noindent\textbf{AMS Subject Classification}: 46N10; 46N30; 49M30; 60A10
 \bigskip

\end{abstract}

\section[]{Introduction}

Minimization of functionals with nonconvex integrands is an important problem both from theoretical and practical, including engineering applications, points of view. The fact that in this case the considered functional, although bounded from below, usually does not attain its infimum, is the source of the main difficulty when dealing with this kind of task. It is described in the following example attributed to Oscar Bolza and Laurence Chisholm Young.
\begin{example} (see \cite{Muller})
Find the minimum of the functional
\[
\mathcal{J}(u)=\int\limits_{0}^{1}\bigl[ u^2+
\bigl( (\tfrac{du}{dx})^2-1\bigr)^2\bigr] dx,
\]
with boundary conditions $u(0)=0=u(1)$.

It can be shown that $\inf\mathcal{J}=0$. 

Consider the function
\[
u(x)=
\begin{cases}
x, 
&\text{for $x\in\big[0,\tfrac{1}{4}\bigr)$}\\
\tfrac{1}{2}-x,
&\text{for $x\in\bigl[\tfrac{1}{4},\tfrac{3}{4}\bigr)$}\\
x-1, &\text{for $x\in\bigl[\tfrac{3}{4},1\bigr)$}.
\end{cases}
\]
Then the sequence $u_n(x):=\tfrac 1 nu(nx)$ is the minimizing sequence for $\mathcal J$, that is it satisfies the condition
$\mathcal J(u_n)\to\inf\mathcal J$. However, we have $\inf\mathcal J\neq\mathcal J(\lim u_n)$ \mbox{because if the limit $u_0$ of $(u_n)$ were the function satisfying the equality $\inf\mathcal J=\mathcal J(u_0)$,} it would have to satisfy simultaneously two mutually contradictory conditions:
$u_0\equiv 0$ and $\tfrac{du_0}{dx}=\pm 1$ a.e. (with respect to the Lebesgue measure on $[0,1]$). This means that $\mathcal{J}$ does not attain its infimum. 
\end{example}
Observe, that the elements of the minimizing sequence are wildly oscillating functions. We will call the functions of this type the \emph{m-oscillating} functions, see Definition \ref{m-oscill}. This is typical situation when one minimizes integral functionals with non-(quasi)convex integrands.

Further motivating examples can be found, among others, in \cite{Pedregal} and \cite{Roubicek}.

Generally, the problem can be dealt with in two ways. The first one is convexification the original functional (more precisely: quasiconvexification). This procedure saves the infimum of the functional, but it has two main drawbacks. Namely, computing explicit form of the (quasi)convex envelope is usually very difficult in practice. Further, it erases some important information concerning the behaviour of the minimizing sequences: calculating a weak$^\ast$ accumulation point of the minimizing sequence is calculation a limit of the sequence with integral elements. The integrands are compositions of the Carath\'eodory function with highly oscillating elements of the minimizing sequence. Thus, in some sense, only the mean values are taken into account, see for example \cite{Puchala3} and references cited there.

Another way is connected with the discovery of Laurence Chisholm Young, first published in \cite{Young1}. He observed that the weak$^\ast$ limit of a sequence with elements being composition of a continuous function with oscillating functions is in fact a \emph{set function}. This set function can be looked at as a mean summarizing the spatial oscillatory properties of minimizing sequences. Thus not all the information concerning behavioral characteristics of the phase involved is lost while passing to the limit. 

More precisely, let there be given: $\mathbb{R}^d\supset\varOmega$ -- nonempty, bounded open set of the Lebesgue measure $M>0$; $K\subset\mathbb{R}^l$ -- compact; $(u_n)$ -- a sequence of functions from 
$\varOmega$ to $K$, convergent to some function $u_0$: weakly$^{\ast}$ in $L^\infty$ or weakly in $L^p$, 
$p\geq 1$; $\varphi$ -- an arbitrary continuous real valued function on $\mathbb{R}^l$.

Then the continuity of $\varphi$ yields the norm boundedness of the sequence $(\varphi(u_n))$ of compositions. By the Banach-Alaoglu theorem we infer the existence of the converging to some function $g$ the subsequence of $(\varphi(u_n))$. However,  in general not only $g\neq\varphi((u_0))$, but $g$ is not even a function with domain in $\mathbb R^l$. It is, as Laurence Chisholm Young first proved for the special case in 1937, a family $(\nu_x)_{x\in\varOmega}$ Borel probability measures defined on Borel $\sigma$-algebra of subsets of $\mathbb R^l$, with supports contained in $K$, satisfying for each continuous $\varphi$ and any integrable function $w$ the condition
\[
\lim\limits_{n\to\infty}\int\limits_{\varOmega}\varphi(u_n(x))w(x)dx=
\int\limits_{\varOmega}\int\limits_K\varphi(s)\nu_x(ds)w(x)dx:=
\int\limits_{\varOmega}\overline{\varphi}(x)w(x)dx,
\]
where
\[
\overline{\varphi}(x):=\int\limits_{\mathbb R^{l}}\varphi(\lambda)\nu_x(d\lambda).
\]
Thus the family $(\nu_x)_{x\in\varOmega}$ can be regarded as the 'generalized weak$^\ast$ limit' of the sequence $(\varphi(u_n))$. Young himself called them 'generalized curves'; today we call them \emph{Young measures} (or \emph{relaxed controls} in control theory) \emph{associated with the sequence} $(u_n)$.
 
In the very important special case, where for any $x\in\varOmega$ there holds $\nu_x=\nu$ (that is the family $(\nu_x)_{x\in\varOmega}$ is a one-element only), we use the term \emph{homogeneous Young measure}. 

We can thus look at Young measure as at the generalized limit of the sequence whose elements are compositions of the fixed continuous function with the elements of weakly (or weakly$^\ast$) convergent sequence of oscillating functions. This approach is exposed in detail in \cite{Pedregal}, where the existence theorem (Theorem 6.2) with very general assumptions is stated and proved; \mbox{one may also consult \cite{Muller}.}

We can also look at the Young measures from another point of view. Namely, they can be regarded as linear functionals acting on 
$L^1(\varOmega,C(K))$ -- the space of vector-valued, Bochner integrable functions defined on $\varOmega$ and taking values in the Banach space of continuous real-valued functions defined on a compact set $K$. This enables us to associate Young measure with \emph{any} Borel function $f\colon\varOmega\to K$, see Theorem 3.1.6 in \cite{Roubicek}. Using this approach one may study Young measures in general, with Young measures associated with measurable functions being specific examples. We refer the reader to \cite{Roubicek} for detailed exposition of this point of view.

Although calculating an explicit form of the Young measure is usually very difficult, it turns out that quite often it can be done relatively easily, when functions under consideration are bounded oscillating ones. This is possible if the latter of the approaches mentioned above is adopted. As it has been observed in \cite{Puchala1} and \cite{GrzybowskiPuchala}, for quite large class of oscillating functions Young measures associated with them are homogeneous, absolutely continuous with respect to the Lebesgue measure with densities being merely the sums of the absolute values of the Jacobians of the inverses of their invertible parts. These sums (called later the \emph{total slopes}, see Definition \ref{tslope}) are important: \emph{different} functions with equal total slopes have the same Young measure. Thus in many cases (for example for the sequence 
$f_n(t)=\sin(2n\pi t),\,t\in(0,1),\,n\in\mathbb N$, a kind of 'canonical' example of the sequence of oscillating functions generating homogeneous Young measure) the sequence of the associated Young measures is constant and it suffices to calculate only the Young measure associated with $f_1$.

In \cite{Puchala2} these observations has been broadened. Namely, it has been proved that the weak convergence of the sequence of homogoneous Young measures (understood as elements of the Banach space $\textnormal{rca}(K)$) with densities is equivalent to the $L^1$ convergence of the sequence of these densities.

It is worth mentioning that the above results have been obtained with really simple, in comparison to the usual Young measure methods, apparatus: the change of variable theorem for multiple integrals plays central role. This can be important in applications.

After the above short presentation of the classical problem of minimizing an integral functional with nonconvex integrand, whose solution leads directly to the Young measures and two approaches to the Young measures, we give outline of the theory needed in the sequel. The third section deals with the sequences of m-oscillating functions, i.e. functions being sums of diffeomorphisms that have continuously differentiable inverses. We introduce the notion of a \emph{total slope} of such oscillating function and observe that the set of the Young measures associated with the elements of the set m-oscillating functions, defined on an open set $\varOmega\subset\mathbb R^d$ with values in the compact set
$K\subset\mathbb R^l$, is relatively weakly compact. The last section begins with recollection of the basic facts concerning weak sequential convergence of functions and measures, in particular the theorem of Jean Dieudonn\'e. Then we introduce the notion of a density of a (not necessarily homogeneous) Young measure. We illustrate this notion with several examples; in particular, one of them makes use of the Bressan-Colombo-Fryszkowski theorem about existence of a continuous selection of a multifunction with decomposable values. We prove that density of a Young measure is a weakly sequentially compact set.

The final part of the fourth section is devoted to homogeneous Young measures and again to m-oscillating functions. First it is shown, that if we additio\-nally assume convexity of the compact set $K$ and consider such Borel functions from $\varOmega$ to $K$, that the Young measures associated with them are homogeneous and absolutely continuous with respect to the Lebesgue measure on $K$, then the norm closedness (in $\textnormal{rca}(K)$ with the total variation norm) of the set of such Young measures is equivalent with the norm closedness (in $L^1(K)$) of the set of their densities. It then follows (already without convexity assumptions on $K$), that the sequence of homogeneous Young measures with densities converges weakly to the measure of the same type if and only if the sequence of the respective densities converges weakly to the density of the limit measure. Finally, given a sequence of m-oscillating functions we show, that if their total slopes form monotonic sequence, than the sequence of the associated Young measures is weakly  convergent to a homogeneous Young measure. Moreover, this limit measure has a density which is a weak limit of the sequence of densities of the respective Young measures. Thus majority of the examples of the homogeneous Young measures generated by the sequences of oscillating functions, that can be met in the literature, turn out to be illustrations \mbox{of the special case of that result.}

\section[]{Necessary facts about Young measures}
In this section we gather information about Young measures that will be of use in the sequel. The literature concerning Young measures is large. An interested reader may consult, for example, besides the books that have already been cited, \cite{Attouch, Castaing, Florescu, Gasinski} and the references cited there. 

In this article we use the approach to Young measures described in \cite{Roubicek} in Chapter 3. 

We will denote by $\textnormal{rca}(K)$  the space of regular, countably additive scalar measures on a Borel $\sigma$-algebra of subsets of a compact set $K$. We equip this space with the norm
$\|m\|_{\textnormal{rca}(K)}:=\vert m\vert(K)$. Here $\vert\cdot\vert$ stands for the total va\-riation of the measure $m$. Then the pair
$(\textnormal{rca}(K),\|\cdot\|_{\textnormal{rca}(K)})$ is a Banach space. The Riesz representation theorem states that in this case
$\textnormal{rca}(K)$ is conjugate space to the space $\big(C(K),\|\cdot\|_\infty\big)$.

Let $(X,\mathcal A,\rho)$ be a measure space, $Y$ -- a Banach space and denote by $\langle\cdot,\cdot\rangle$ the duality pairing. Recall, that a function $f\colon X\to Y^\ast$ is \emph{weakly$^\ast$-measurable} if for any $y\in Y$ the function
\[
x\to\langle f(x),y\rangle
\]
is $\mathcal A$-measurable.

The elements of a space $L_{w^\ast}^{\infty}(\varOmega;\textnormal{rca}(K))$ are the functions 
\[
\nu\colon\varOmega\ni x\to\nu(x)\in\textnormal{rca}(K),
\]
that are weakly$^\ast$-measurable and such that
\[\textnormal{ess}\sup\bigl\{\Vert\nu (x)\Vert_{\textnormal{rca}(K)}:
x\in\varOmega\bigr\}<+\infty,
\]
where
\[
\textnormal{ess}\sup\bigl\{\Vert\nu (x)\Vert_{\textnormal{rca}(K)}:
x\in\varOmega\bigr\}:=
\inf\big\{\alpha\in\mathbb R\cup\{\infty\}:\Vert\nu (x)\Vert_{\textnormal{rca}(K)}\leq\alpha\;\textnormal{a.e. in}\,\varOmega\big\}.
\]
We endow this space with a norm
\[
\Vert\nu\Vert_{L_{w^\ast}^{\infty}(\varOmega ,\textnormal{rca}(K))}:=
\textnormal{ess}\sup\bigl\{\Vert\nu (x)\Vert_{\textnormal{rca}(K)}:
x\in\varOmega\bigr\}.
\]
It is proved in \cite{Roubicek} that the space 
$\big(L^1(\varOmega,C(K))\big)^\ast$ is isometrically isomorphic to the space $L_{w^\ast}^{\infty}(\varOmega;\textnormal{rca}(K))$.

The set $\mathcal Y(\varOmega,K)$ of the \emph{Young measures} consists of those functions from the space 
$\big(L_{w^\ast}^{\infty}(\varOmega;\textnormal{rca}(K)), \Vert\nu\Vert_{L_{w^\ast}^{\infty}(\varOmega ,\textnormal{rca}(K))}\big)$, whose values are probability measures on $K$. 
The Theorem 3.1.6 allows one to conclude that for \emph{any} measurable function
$f\colon\varOmega\to K$ there exists a Young measure $\nu^f$ associated with this function.

A particularly simple, yet important, example of a Young measure is a \emph{homogeneous Young measure}. This is a case, when 'the family of probability measures 
$\nu=(\nu_x)_{x\in\varOmega}$ reduces to a unique single measure: $\nu_x=\nu$ for a.e 
$x\in\varOmega$', to cite from page 9 of \cite{Pedregal}. In this case the set 
$\mathcal Y(\varOmega,K)$ is a subset of the set of probability measures on $K$. 
Majority of the specific examples of the Young measures that can be found in the literature are homogeneous ones, see for example \cite{Florescu}, Example 3.44; \cite {Muller}, paragraphs 3.2, 3.3, 4.6, 6.2; \cite{Pedregal}, second and third paragraph of the first chapter, second and fourth paragraphs of chapter 9; \cite{Roubicek} pages 116, 117. 

Let $\varOmega\subset\mathbb R^d$ be a bounded open set with the Lebesgue measure $dx$ and let $M>0$ be a Lebesgue measure of $\varOmega$. Define
$d\lambda(x):=\tfrac 1 Mdx$ and consider a compact set $K\subset\mathbb R^l$ with the Lebesgue measure $d\mu$ and a Borel measurable function $f\colon\varOmega\to K$. Due to the fact that homogeneous Young measures are 'one-element families', we will write '$\nu$' instead of 
'$\nu=(\nu_x)_{x\in\varOmega}$' for such measures and denote by '$\nu^f$' the homogeneous Young measure associated with the function $f$.

According to the Convention 3.1.1, the Theorem 3.1.6 in \cite{Roubicek} and the Theorem 3.1 in \cite{Puchala2} we will use the following definition of the homogeneous Young measure.
\begin{definition}
\begin{itemize}
\item[(i)]
we say that a mapping $\nu\in\mathcal Y(\varOmega,K)$ is a \emph{homogeneous Young measure} if it is constant on $\varOmega$;
\item[(ii)]
let $\nu^f$ be a Young measure associated with a Borel function $f\colon\varOmega\to K$. We say that $\nu^f$ is a \textnormal{homogeneous Young measure}  if it is constant on $\varOmega$ and is an image of the measure $\lambda$ under $f$, i.e. $\nu^f=\lambda\circ f^{-1}$.
\end{itemize}
\end{definition}

\section[]{Oscillating functions}\label{oscfun}
Let $\varOmega\subset\mathbb R^d$ be a nonempty, bounded open set of the Lebesgue measure $M>0$. Consider $\{\varOmega\}$ -- an open partition of $\varOmega$ into at most countable number of open subsets $\varOmega_1,\varOmega_2,\dots,\varOmega_n,\dots$ such that
\begin{itemize}
\item[(i)]
the elements of $\{\varOmega\}$ are pairwise disjoint;
\item[(ii)]
$\bigcup\limits_{i}\overline{\varOmega}_i=\overline{\varOmega}$, where $\overline{A}$ denotes the closure of the set $A$.
\end{itemize}

Let us consider   functions $f_i\colon\varOmega_i\to K\subset\mathbb{R}^d$, $i=1,2,...$, with inverses $f_i^{-1}$  that are  continuously differentiable on $f(\varOmega_i)$ and let $K_i:=\overline{f_i(\varOmega_i)}$ be compact. Denote for each  $i=1,2,...$ the Jacobian of $f_i^{-1}$ by $J_{f_i^{-1}}$.
\begin{definition}\label{m-oscill}
We say that a function  $f\colon\varOmega\to K$, with $K:=\overline{f(\varOmega)}$ compact, is an \textnormal{m-oscillating function} if it is of the form
\begin{equation}\label{uncountable}
f(x)=\sum\limits_{i} f_i(x)\chi_{\varOmega_i}(x). 
\end{equation}
\end{definition}
\begin{remark}
The letter 'm' in the above definition refers to the fact, that in one dimensional case both the functions $f_i$ and $f_i^{-1}$ are strictly monotonic ones. This is to distinguish from the case when oscillating functions are piecewise constant, for example the Rademacher functions.
\end{remark}
Then there holds the following theorem (see Proposition 5.1. in \cite{GrzybowskiPuchala}).
\begin{theorem}\label{Prop}
The Young measure associated with function $f$ satisfying (\ref{uncountable}) is a homogeneous one and its density $g$  with respect to the Lebesgue measure on $K$ is given by the formula
\begin{equation}\label{density}
g(y)=\frac 1 M\sum\limits_{i:y\in K_i }\vert J_{f_i^{-1}}(y)\vert.
\end{equation}
\end{theorem}
The above result suggests introducing the following notion.
\begin{definition}\label{tslope}
Let an oscillating function $f$ be given by the equation (\ref{uncountable}). The 
\textnormal{total slope} $Jt_f$ of $f$ is defined by
\[
Jt_f(y):=\sum\limits_{i:y\in K_i }\vert J_{f_i^{-1}}(y)\vert.
\]
\end{definition}
\begin{example}\label{sinus}
\begin{itemize}
\item[(i)]
consider an example on page 117 in \cite{Roubicek} of a sequence $(f_n)$ of m-oscillating nonperiodic functions, where for each
$n\in\mathbb N$ we have
\[
f_n(x):=
\begin{cases}
\big(x(n+k-1)-k+1\big)\tfrac{n+k}{n}, & x\in
\Big(\tfrac{k-1}{n+k-1},\tfrac{k}{n+k}\Big),\; k\in\mathbb{N}\;\textnormal{odd}\\
\big(k-x(n+k)\big)\tfrac{n+k-1}{n}, & 
x\in\Big[\tfrac{k-1}{n+k-1},\tfrac{k}{n+k}\Big),\; k\in\mathbb{N}\;\textnormal{even}.
\end{cases}
\]
Then the sequence $(Jt_{f_n})$ of the respective total slopes is constant with each element equal to 1. It means that the Young measure associated with each $f_n$ (and therefore the Young measure generated by the sequence $(f_n)$) is a homogeneous one, absolutely continuous with respect to the Lebesgue measure on $[0,1]$ with density that is equal to 1 a.e;
\item[(ii)]
let $f_n(t)=\sin(2n\pi t),\,t\in(0,1),\,n\in\mathbb N$. Then 
\[
\forall n\in\mathbb N\quad Jt_{f_n}(y)=\tfrac{1}{\pi\sqrt{1-y^2}},
\]
and thus the Young measure associated with each $f_n$ is homogeneous and absolutely continuous with respect to the Lebesgue measure. Its density is equal to $Jt_{f_n}$.
\end{itemize}
\end{example}
The following result follows from Theorem 1.64 in \cite{Florescu} and Theorem \ref{Prop}.
\begin{theorem}\label{oscill-rwc}
Consider the set $\mathcal A$ of m-oscillating functions defined on a nonempty, bounded open set 
$\varOmega\subset\mathbb R^d$ of positive Lebesgue measure, having values in a compact set $K\subset\mathbb R^d$ . Then the set of the Young measures associated with the elements of $\mathcal A$ is relatively weakly compact.
\end{theorem}
\section[]{Weak sequential convergence of functions and measures}
We will now recall classical theorems concerning weak sequential convergence of functions and measures that will be needed in the sequel. The expression 'weak convergence of the sequence of measures' will be meant as the 'weak convergence of the sequence of measures as elements of the Banach space $\textnormal{rca}(K)$' (with the total variation norm). 

Recall that if $\rho$ is a measure on $K$ and for some function $w\colon K\to\mathbb{R}$ integrable with respect to the measure $\xi$ there holds: for any Borel subset $A$ of $K$ we have $\rho(A)=\int_Aw(y)d\xi(y)$, then the function $w$ is called a \emph{density} of the measure $\rho$. In this case $\rho$ is \emph{absolutely continuous} with respect to $\xi$ (shortly: $\xi$-continuous): $\xi(A)=0\Rightarrow\rho(A)=0$.

Let $(X,\mathcal A,\rho)$ be a measure space and consider a sequence $(u_n)$ of scalar functions defined on $X$ and integrable with respect to the measure $\rho$ (that is, $\forall n\in\mathbb N\; f_n\in L^1_\rho(X)$) and a function $u\in L^1_\rho(X)$. Recall that $(u_n)$ converges weakly sequentially to $u$ if
\[
\forall g\in L^\infty_\rho(X)\quad \lim\limits_{n\to\infty}\int\limits_Xu_ngd\rho=\int\limits_Xugd\rho.
\]
The following theorem characterizes weak sequential $L^1$ convergence of functions and weak convergence of measures. We refer the reader to \cite{Florescu, Czaja, Diestel}.
\begin{theorem}\label{CzajaDiestel}
\begin{itemize}
\item[(a)](J. Dieudonn\'e, 1957) 
let $X$ be a locally compact Hausdorff space and $(X,\mathcal A,\rho)$ -- a measure space with $\rho$ regular. A sequence 
$(u_n)\subset L^1_\rho(X)$ converges weakly to some $u\in L^1_\rho(X)$ if and only if\, $\forall A\in\mathcal A$ the limit
\[
\lim\limits_{n\to\infty}\int\limits_Au_nd\rho
\]
exists and is finite;
\item[(b)]
let $X$ be a locally compact Hausdorff space and denote by $\mathcal B(X)$ the $\sigma$-algebra of Borel subsets of $X$. A sequence $(\rho_n)$ of scalar measures on $\mathcal B(X)$ converges weakly to some scalar measure $\rho$ on $\mathcal B(X)$ if and only if\, $\forall A\in\mathcal B(X)$ the limit
\[
\lim\limits_{n\to\infty}\rho_n(A)
\]
\nopagebreak exists and is finite.
\end{itemize}
\end{theorem}
\begin{corollary}\label{weakconvmeas}
\begin{itemize}
\item[(a)]
let $(\rho_n)$ be a sequence of measures having respective densities $u_n$, $n\in\mathbb N$. Then the sequence $(u_n)$ is weakly convergent in $L^1(X)$ to some function $h$ if and only if the sequence $(\rho_n)$ is weakly convergent to some measure $\eta$;
\item[(b)]
assume additionally, that $X\subset\mathbb R^l$ is compact and let $(\rho_n)$ be a sequence of homogeneous Young measures having respective densities $u_n$. Then the sequence $(u_n)$ is weakly convergent in $L^1(X)$ to some function $h$ if and only if the sequence $(\rho_n)$ is weakly convergent to some measure $\eta$.
\end{itemize}
\end{corollary}
We now introduce the notion of a density of a Young measure. We remember, that Young measure is in fact a family $\nu=(\nu_x)_{x\in\varOmega}$. In a special case, when this family consists of one element only (i.e. it does not depend on the variable $x$) we say about homogeneous Young measure.

Let $\varOmega$ be a nonempty, bounded open subset of $\mathbb R^n$, $K$ -- compact subset of $\mathbb R^l$. 
The Borel $\sigma$-algebra of subsets of $K$ will be denoted $\mathcal B(K)$.
\begin{definition}\label{ymdens}
We say that a family $h=(h_x)_{x\in\varOmega}$ is a density of a Young measure $\nu$ with respect to the measure $\xi$ defined on 
$\mathcal B(K)$ if for any $x\in\varOmega$ the function $h_x$ is a density of the measure $\nu_x$ i.e. for any $A\in\mathcal B(K)$ there holds $\nu_x(A)=\int_Ah_x(y)d\xi(y)$.
\end{definition}
The next result follows directly from the above definition.
\begin{proposition}\label{propymdens}
Let $\nu$ be a Young measure and let $h=(h_x)_{x\in\varOmega}$ be a density of $\nu$ with respect to the measure $\xi$. Then $\nu$ is a homogeneous Young measure if and only if the family $h=(h_x)_{x\in\varOmega}$ consists of one element only, up to the set of $\xi$-measure $0$. 
\end{proposition}
\begin{example}
\begin{itemize}
\item[(i)]
consider a function $f\colon(0,1)\to[0,1]$ given by the formula
\[
f(x):=
\begin{cases}
2x, & \textnormal{if}\,x\in\big(0,\tfrac 1 2\big)\\
\frac 1 2 & \textnormal{if}\,x\in\big[\tfrac 1 2,1\big).
\end{cases}
\]
Then the Young measure associated with $f$ has no density;
\item[(ii)]
the density of the Young measure associated with each element of the sequence $(f_n)$, where 
$f_n(x):=\sin(2n\pi x),\;x\in(0,1)$, is a 'one element family'  $h=(h_x)_{x\in\varOmega}$, with 
$h(y)=\tfrac{1}{\pi\sqrt{1-y^2}}$;
\item[(iii)]
let $\varOmega=(0,1)$, $K=[0,2]$ and denote by $\mu$ a Lebesgue measure on $\mathcal B(K)$. Consider a family $h=(h_x)_{x\in\varOmega}$ of functions defined as follows:
for each $x\in\varOmega$ and $y\in K$
\[
h_x(y):=
\begin{cases}
\frac 1 xy & \textnormal{if}\,y\in[0,x)\\
-\frac{1}{1-x}y+\frac{1}{1-x} & \textnormal{if}\,y\in[x,1).
\end{cases}
\]
Then $h=(2h_x)_{x\in\varOmega}$ is a density of a nonhomogeneous Young measure 
$\nu=(\nu_x)_{x\in\varOmega}$, where for each $x\in\varOmega$ we have $\nu_x:=2h_x\mu(dy)$;
\item[(iv)]
for this generalization of the above example we need some of the basic notions of the set-valued analysis. All the necessary set-valued theory can be found, for example, in \cite{Denkowski, Fryszko, Hu, Repovs}.

Let $\varOmega$ and $K$ be given nonempty sets.
By a \textnormal{multivalued mapping} or \textnormal{multifunction} we mean a mapping \mbox{$T\colon\varOmega\to 2^K$,} often denoted by \mbox{$T\colon\varOmega\leadsto K$.}  For given set $A\subset K$ denote 
$T^-(A):=\{x\in\varOmega: T(x)\cap A\neq\emptyset\}$. We say that a multifunction $T$ is \textnormal{lower semicontinuous} at $x_0\in\varOmega$ if and only if for every open set $V\subset K$ such that $x_0\in T^-(V)$, $x_0$ is an interior point of $T^-(V)$. A \textnormal{selection} of a multifunction $T$ is a single valued function $t\colon\varOmega\to K$ such that for all $x\in\varOmega$ there holds $t(x)\in T(x)$. By the Axiom of Choice every multifunction has a selection, but one of the main problems of set-valued analysis is investigation of the existence of selections having certain regularity properties, like measurability or continuity.

We say that a set $D\subset L^1(K)$ is \textnormal{decomposable}, if for all $u,v\in D$ and all 
$A\in\mathcal B(K)$ the function $\chi_A\cdot u+\chi_{K\setminus A}\cdot v$ is an element of $D$; the symbol $\chi_A$ stands for the characteristic function of the set $A$.

Let $\varOmega=[0,1]$, $K=[0,1]$, $\mu$ -- a Lebesgue measure on $\mathcal B(K)$ and let $T\colon\varOmega\leadsto L^1(K)$ be a lower semiconti\-nuous multifunction having nonempty, closed and decomposable values. By the Bressan-Colombo-Fryszkowski theorem, see for instance \cite{Fryszko}, Theorem 42 or \cite {Repovs}, Theorem (7.18), there e\-xists a continuous selection for $T$, that is a continuous function $t\colon\varOmega\to L^1(K)$ such that for each $x\in\varOmega$ there holds $t(x)\in T(x)$. The family $h=(h_x)_{x\in\varOmega}$, 
where $h_x:=t(x)$ and $t(x)\in L^1(K)$ is $\mu$-a.e. nonzero, gives rise to the density of the nonhomogeneous Young measure.
\end{itemize}
\end{example}
We will now recall two classical theorems that will be needed in the sequel. We again refer the reader to \cite{Florescu, Czaja, Diestel}.
\begin{theorem}
\begin{itemize}
\item[(a)] $($Radon-Nikodym theorem for finite measures$)$\\
let $\mu$ and $\rho$ be finite measures on a measurable space $(X,\mathcal A)$ and assume that $\rho$ is $\mu$-continuous. Then there exists a unique $($up to the set of $\mu$-measure $0$$)$ $\mu$-integrable function $w\in L^1(X)$ such that
\[
\forall A\in\mathcal A\quad\rho(A)=\int\limits_Awd\mu;
\]
\item[(b)] $($Vitali-Hahn-Saks theorem$)$\\
let $(X,\mathcal A)$ be measurable space, $\mu$ -- a nonnegative finite measure and let $(\rho_n)$ be a sequence of $\mu$-continuous scalar measures on $\mathcal A$. If for any $A\in\mathcal A$ the limit $\lim\limits_{n\to\infty}\rho_n(A)$ exists, then the formula:
\[
\forall A\in\mathcal A\quad \rho(A):=\lim\limits_{n\to\infty}\rho_n(A)
\]
defines a $\mu$-continuous scalar measure on $\mathcal A$.
\end{itemize}
\end{theorem} 
Let the family $(h_x)_{x\in\varOmega}$ be a density of a Young measure 
$(\nu_x)_{x\in\varOmega}$. We then have the following theorem.

\newpage

\begin{theorem}\label{conv_nonhom_ym}
Let $(\varOmega,\textnormal{dist})$ be a metric space and assume that the mapping
\[
h\colon\varOmega\ni x\to h(x):=h_x\in L^1(K)
\]
is $(\varOmega,\textnormal{dist})$ -- weakly-sequentially-in-$L^1(K)$ continuous. Let $(x_n)$ be a sequence in $\varOmega$ convergent to $x_0\in\varOmega$. Then the sequence $(h_{x_n})$ converges weakly to $h_{x_0}$ if and only if the sequence 
$(\nu_{x_n})$ converges weakly to the measure $\nu_{x_0}$ having density $h_{x_0}$.
\end{theorem}
\begin{proof}
The first part of the theorem follows from part (b) of the Corollary \ref{weakconvmeas}. Now let the sequence $(\nu_{x_n})$ converges weakly to the measure 
$\eta$. It follows from the Vitali-Hahn-Saks theorem that the measure $\eta$ is $\mu$-continuous, so by the Radon-Nikodym theorem it has a density $r$. Choose and fix set $A\in\mathcal B(K)$. The fact that $r=h_{x_0}$ up to the set of $\mu$-measure $0$ follows from arbitrariness of the choice of $A$ and the inequality 
\begin{multline}
\Big\vert\int\limits_{A}r(y)d\mu-\int\limits_{A}h_{x_0}(y)d\mu\Big\vert\leq
\Big\vert\int\limits_{A}r(y)d\mu-\eta(A)\Big\vert+\Big\vert\eta(A)-\nu_{x_n}(A)\Big\vert+\\ \notag
+\Big\vert\nu_{x_n}(A)-\int\limits_{A}h_{x_n}(y)d\mu\Big\vert+
\Big\vert\int\limits_{A}h_{x_n}(y)d\mu-\int\limits_{A}h_{x_n}(y)d\mu\Big\vert,
\end{multline}
which in turn implies that $\eta=\nu_{x_0}$.
\end{proof}
The next result follows from the above theorem and weak sequential completness of $L^1(K)$ (Theorem 1.3.13 in \cite{Lin}).
\begin{corollary}
The density $h=(h_x)_{x\in\varOmega}$ of a Young measure $\nu=(\nu_x)_{x\in\varOmega}$ is a weakly sequentially closed set.
\end{corollary}
We now turn to the case of homogeneous Young measures associated with Borel functions.
Denote by $\mu$ the Lebesgue measure on $\mathcal B(K)$, by $\mathcal M$ -- the set of all Borel measurable functions from $\varOmega$ to $K$, by $\mathcal M_H\subset \mathcal M$ -- the set of functions such that the Young measures associated with them are homogeneous and $\mu$-continuous,  and define
\begin{itemize}
\item
\[
\mathcal F:=\big\{\nu^h\in\mathcal Y(\varOmega,K):h\in\mathcal M_H\big\};
\]
\item
\[
\mathcal D:=\big\{u\colon K\to\mathbb R: u\;\textnormal{is the density of the measure from}\;\mathcal F\big\}.
\]
\end{itemize}
Recall that for convex subsets of a normed space strong and weak closures coincide, see for example Proposition A.3.19 in \cite{Gasinski}. We thus have the following theorem.
\begin{theorem}\label{weakconvYM}
Assume that the set $K$ is convex. Then the set $\mathcal F$ is closed in the norm topology of 
$\big(\textnormal{rca}(K),\|\cdot\|_{\textnormal{rca}(K)}\big)$ if and only if the set $\mathcal D$ is closed in the norm topology of the space $L^1(K)$.
\end{theorem}
\begin{proof}
We use part (b) of the Corollary \ref{weakconvmeas}. The set $K$ is convex so are 
$\mathcal M_H$, $\mathcal F$ and $\mathcal D$. 
Assume that the set $\mathcal F$ is weakly closed and let $u_n$ be a sequence from $\mathcal D$ convergent weakly to $u_0$. Then the sequence $(\nu_n)$, where $\nu_n=\nu^{u_n}=u_nd\mu$, converges weakly to $\nu_0\in\mathcal F$. Assumptions that the density $w$ of $\nu_0$ is not equal $\mu$-almost everywhere to $u_0$ and that 
$u_0\notin\mathcal D$ lead to a contradiction, proving that the set $\mathcal D$ is weakly closed in $L^1(K)$, hence strongly closed.

Conversely, assume that the set $\mathcal D$ is weakly closed and let $(\rho_n)$ be a sequence of the Young measures from the set $\mathcal F$, weakly convergent to the measure $\eta$. Denote their respective densities by $u_n$. Again, from the Vitali-Hahn-Saks theorem and the Radon-Nikodym theorem we infer respectively, that the measure $\eta$ is $\mu$-continuous and has a density $r$. The fact that $r=h$ up to the set of $\mu$-measure $0$ is proved analogously like in the Theorem \ref{conv_nonhom_ym}, while the inequality
\[
\Big\vert1-\int\limits_{K}hd\mu\Big\vert\leq\Big\vert1-\int\limits_{K}u_nd\mu\Big\vert+
\Big\vert\int\limits_{K}u_nd\mu-\int\limits_{K}hd\mu\Big\vert
\]
shows, that $\eta$ is a probability measure on $K$. 

The uniqueness of the weak limit gives the weak$^\ast$-measurability of the mapping
\[
\varOmega\ni x\to\eta\in\textnormal{rca}(K),
\]
proving that $\eta$ is a Young measure. Its homogenity is a consequence of the Proposition \ref{propymdens}.
\end{proof}
Respective weak convergences of Young measures and their densities still hold without convexity assumptions on $K$.
\begin{corollary}\label{concl-weakconvYM}
Let $K\subset\mathbb R^l$ be compact. Let $(\rho_n)$ be a sequence of  $\mu$-continuous homogeneous Young measures with respective densities $u_n$.\\
Then the sequence $(u_n)$ is weakly convergent in $L^1(K)$ to some function $h$ if and \mbox{only if the sequence $(\rho_n)$ is weakly convergent in the Banach space $(\textnormal{rca}(K),\|\cdot\|_{\textnormal{rca}(K)})$} to a homogeneous Young measure $\eta$ with density $h$.
\end{corollary}
Consider again examples from pages 116 and 117 in \cite{Roubicek} of the sequences of m-oscillating functions and the Young measures they generate. In \cite{Puchala1} it is shown, that the Young measures associated with each element of the particular sequence form  constant, hence trivially weakly$^\ast$ convergent, sequence. It follows from the fact, that the sequence of the total slopes of these m-oscillating functions is constant. It is a special case of a more general situation described in the next result. It originates from a simple observation in \cite{Puchala4}.
\begin{theorem}
Consider the sequence $(f_n)$ of m-oscillating functions. Assume further that the sequence $(Jt_{f_n})$ of respective total slopes is monotonic. Then the sequence $(\rho_n)$ of Young measures associated with the functions $f_n$ is weakly convergent to the homogeneous Young measure $\rho$. Moreover, the measure $\rho$ has density which is equal to the weak $L^1$ limit of the sequence $(u_n)$ of densities of the Young measures associated with the functions $f_n$.
\end{theorem}
\begin{proof}
We can assume that the sequence $(Jt_{f_n})$ is nondecreasing. \mbox{By Theorem \ref{Prop}} the Young measures $\rho_n$ are homogeneous with densities given by the \mbox{equation (\ref{density}).} Then for any $m,\,n\in\mathbb N$, $m\leq n$ and any $A\in\mathcal B(K)$ there holds an inequality
\[
\int\limits_AJt_{f_m}d\mu\leq\int\limits_AJt_{f_n}d\mu,
\]
which, together with Theorem \ref{CzajaDiestel} and Corollary \ref{concl-weakconvYM}, gives the result.
\end{proof}

\end{document}